\documentclass{amsart}[12pt]
\usepackage{amssymb,amsmath}
\usepackage{enumerate}
\usepackage[english]{babel}
\usepackage{color}

\newcommand {\ep} {\varepsilon}

\newcommand {\ii} {\infty}

\newcommand {\al} {\alpha}

\newcommand {\lb} {\lambda}

\newcommand {\sm} {\setminus}
\newcommand {\su} {\subset}

\newcommand {\wh} {\widehat}

\newcommand {\mc} {\mathcal}

\newtheorem{teo}{Theorem}[section]
\newtheorem{pro}{Proposition}[section]
\newtheorem{cor}{Corollary}[section]

\newtheorem{lm}{Lemma}[section]

\theoremstyle{definition}
\newtheorem{rem}{Remark}[section]

\title{Ergodic theorems \\ in  Banach ideals of compact operators}
\keywords{Symmetric sequence space, Banach ideal of compact operators, Dunford-Schwartz operator, individual ergodic theorem,  mean ergodic theorem}
\subjclass[2010]{46E30, 37A30, 47A35}

\begin{document}
\date{February 20, 2019}

\begin{abstract}
Let $\mc H$ be an infinite-dimensional Hilbert space, and let $\mc B(\mc H)$ ($\mc K(\mc H)$) be the $C^*$-algebra of bounded (respectively, compact) linear operators in $\mc H$. Let $(E,\|\cdot\|_E)$ be a fully symmetric sequence space. If $\{s_n(x)\}_{n=1}^\ii$ are the singular values of $x\in\mc K(\mc H)$, let $\mc C_E=\{x\in\mc K(\mc H): \{s_n(x)\}\in E\}$ with $\|x\|_{\mc C_E}=\|\{s_n(x)\}\|_E$, $x\in\mc C_E$, be the Banach ideal of compact operators generated by $E$. We show that the averages $A_n(T)(x)=\frac1{n+1}\sum\limits_{k = 0}^n T^k(x) $ converge uniformly in $\mc C_E$ for any Dunford-Schwartz operator $T$ and $x\in\mc C_E$. Besides, if $x\in\mc B(\mc H)\sm\mc K(\mc H)$, there exists a Dunford-Schwartz operator $T$ such that the sequence $\{A_n(T)(x)\}$ does not converge uniformly. We also show that the averages $A_n(T)$ converge strongly in $(\mc C_E, \|\cdot\|_{\mc C_E})$ if and only if $E$ is separable and $E \neq l^1$, as sets.
\end{abstract}

\author{A. Azizov}
\address{National University of Uzbekistan\\
Tashkent, 100174, Uzbekistan}
\email{ azizov.07@mail.ru }
\author{V. Chilin}
\address{National University of Uzbekistan\\
Tashkent, 100174, Uzbekistan}
\email{vladimirchil@gmail.com; chilin@ucd.uz}
\author{S. Litvinov}
\address{Pennsylvania State University \\
76 University Drive \\ Hazleton, PA 18202, USA}
\email{snl2@psu.edu}

\maketitle

\section{Introduction}

Let $\mc B(\mc H)$ be the algebra of bounded linear operators in a Hilbert space $\mc H$, equipped the uniform norm $\|\cdot\|_\ii$. The study of noncommutative individual ergodic theorems in the space of measurable operators affiliated with a semifinite von Neumann algebra
$\mc M\su\mc B(\mc H)$ equipped with a faithful normal semifinite trace $\tau$ was initiated by F. Yeadon.  In \cite{ye}, as a corollary of a noncommutative maximal ergodic inequality in $L^1=L^1(\mc M,\tau)$, the following individual ergodic theorem was established.

\begin{teo}\label{t11}
Let $T: L^1 \to L^1$ be a positive $L^1-L^\ii$-contraction. Then for any $x\in  L^1$ there exists $\wh x\in  L^1$ such that the averages
\begin{equation}\label{e11}
A_n(T)(x)=\frac1{n+1}\sum\limits_{k = 0}^n T^k(x)
\end{equation}
 converge to $\wh x$ bilaterally almost uniformly (in Egorov's sense), that is, given $\ep>0$, there exists a projection $e\in\mc M$ such that $\tau(\mathbf 1-e)<\ep$ and
\[
\|e( A_n(T)(x) -\wh x)e \|_\ii \to 0 \text{ \ \ as\ \ } n\to\ii,
\]
where $\mathbf 1$ is the unit of $\mc M$.
\end{teo}

The study of individual ergodic theorems beyond $L^1(\mc M,\tau)$ started much later with another fundamental paper by M. Junge and Q. Xu \cite{jx}, where, among other results, individual ergodic theorem was extended to the case with a positive Dunford-Schwartz operator acting in the space $L^p(\mc M ,\tau)$, $1<p<\ii$. In \cite{cl1} (\cite{cl2}), utilizing the approach of \cite{li}, an individual ergodic theorem was proved for a positive Dunford-Schwartz operator in a noncommutative Lorentz (respectively, Orlicz) space.

Let $\mc H$ be an infinite-dimensional Hilbert space. Let $E\su c_0$ be a fully symmetric sequence space. Denote by $\mc C_E$  the Banach ideal of compact operators in $\mc H$ associated with $E$. In Section 3 of the article, we obtain the following individual Dunford-Schwartz-type ergodic theorem.

\begin{teo}\label{t12}
\begin{enumerate}[(i)]
\item
Given a  Dunford-Schwartz operator $T:\mc C_E\to\mc C_E$ and $x\in\mc C_E$, there exists $\wh x\in\mc C_E $ such that
$\|A_n(T)(x) -\wh x\|_\ii\to 0$ as $n\to\ii$;
\item
if $x\in\mc B(\mc H)\sm\mc K(\mc H)$, then there exists a  Dunford-Schwartz operator $T:\mc B(\mc H)\to \mc B(\mc H)$ such that the averages $A_n(T)(x)$ do not converge uniformly.
\end{enumerate}
\end{teo}

Noncommutative mean ergodic theorem can be stated as follows: if $T$ is an $L^1-L^\ii$-contraction and $1<p<\ii$, then the averages $A_n(T)$ converge strongly in $L^p=L^p(\mc M,\tau)$, that is, given $x\in L^p$, there exists $\wh x\in L^p$ such that $\|A_n(T)(x)-\wh x\|_p \to 0$ as $n\to\ii$. If $p=1$ and $\tau(\mathbf 1)=\ii$, this is not true in general. As a consequence, if $\tau(\mathbf 1)=\ii$, mean ergodic theorem may not hold in some noncommutative symmetric spaces.
In Yeadon's paper \cite{ye1}, the following mean ergodic theorem was established.

\begin{teo}\label{t13}
Let $E=(E(\mc M,\tau),\|\cdot\|_E)$ be a noncommutative fully symmetric space such that
\begin{enumerate}[(i)]
\item
$L^1(\mc M,\tau)\cap \mc M$ is dense in $E$;
\item
$\|e_n\|_E\to 0$ for any sequence of projections $\{e_n\}\su L^1(\mc M,\tau)\cap\mc M$ with $e_n\downarrow 0$;
\item
$\|e_n\|_E/\tau(e_n) \rightarrow 0$ for any increasing sequence of projections $\{e_n\}\su\mc M$, \\$0<\tau(e_n)<\ii$, with
$\tau(e_n)\to \ii$.
\end{enumerate}
Then for any $x\in E$ and a positive $L^1-L^\ii$-contraction $T: E\to E $ there exists $\wh x \in E$ such that
$\|A_n(T)(x)-\wh x\|_E\to 0$.
\end{teo}

In \cite{cl1}, a mean ergodic theorem was established for a noncommutative symmetric space $E(\mc M,\tau)$ associated with a fully symmetric function space with nontrivial Boyd indices and order continuous norm.

In Section 4, we give the following criterion for the validity of the mean ergodic theorem in a Banach ideal of compact operators in $\mc H$.

\begin{teo}\label{t14} The following conditions are equivalent:
\begin{enumerate}[(i)]
\item
For any Dunford-Schwartz operator $T:\mc C_E\to\mc C_E$ the averages $A_n(T)$ converge strongly in $\mc C_E$;
\item
$(E, \|\cdot\|_E)$ is separable and $E\neq l^1$, as sets.
\end{enumerate}
\end{teo}

Commutative counterparts of Theorems \ref{t12} and \ref{t14} were established in \cite{ca}.

In the last section of the article, we present applications of Theorems \ref{t12} and \ref{t14} to the well-studied Orlicz, Lorentz, and Marcinkiewicz  ideals of compact operators.

\section{Preliminaries}

Let $l^\ii$ (respectively,  $c_0$)  be the Banach lattice of bounded (respectively, converging to zero) sequences $\{\xi_n\}_{n=1}^\ii$ of complex numbers equipped with the norm $\|\{\xi_n\}\|_\ii=\sup\limits_{n \in \mathbb N} |\xi_n|$, where $\mathbb N$ is the set of natural numbers. If $2^{\mathbb N}$ is the $\sigma$-algebra of subsets of $\mathbb N $ and
$\mu(\{n\})=1$ for each $n\in\mathbb N$, then $(\mathbb N, 2^{\mathbb N},\mu)$ is a $\sigma$-finite measure space such that $L^\ii(\mathbb N, 2^{\mathbb N},\mu)=l^\ii$ and
\[
L^1(\mathbb N, 2^{\mathbb N},\mu)=l^1=\left \{\{ \xi_n \}_{n = 1}^\ii\su\mathbb C:\ \|\{ \xi_n \} \|_1
=\sum_{n=1}^\ii |\xi_n|<\ii\right \}\subset l^\ii,
\]
where $\mathbb C$ is the set of complex numbers.

If $\xi=\{\xi_n\}_{n = 1}^\ii\in l^\ii$, then the {\it non-increasing rearrangement} $\xi^*:(0,\ii) \to (0,\ii)$ of $\xi$ is defined by
\[
\xi^*(t)=\inf\{\lb:\mu\{|\xi|>\lb\}\leq t\},\ \ t > 0,
\]
(see, for example, \cite[Ch.\,2, Definition 1.5]{bs}). As such, the non-increasing rearrangement of
a sequence $\{\xi_n\}_{n = 1}^\ii \in l^\ii$ can be identified with the sequence $\xi^*=\{\xi_n^*\}_{n=1}^\ii$, where
\[
\xi_n^*= \inf\left\{ \sup\limits_{n \notin F} |\xi_n |: F\su \mathbb N,\ |F|<n\right\}.
\]
If $\{\xi_n\}\in c_0$, then $\xi_n^* \downarrow 0$; in this case there exists a bijection $\pi:\mathbb N \rightarrow \mathbb N$ such that $|\xi_{\pi(n)}| = \xi_n^*$, $n\in\mathbb N$.

{\it Hardy-Littlewood-Polya partial order} in the space  $l^\ii$ is defined as follows:
\[
\xi=\{\xi_n\}\prec\prec\eta=\{\eta_n\}\ \ \Longleftrightarrow \ \ \sum_{n=1}^m \xi^*_n \leq \sum_{n=1}^m \eta^*_n \ \text{\ \ for all\ \ } \ m\in\mathbb N.
\]

A non-zero linear subspace $E\su l^\ii$ with a Banach norm $\|\cdot\|_E$ is called a {\it symmetric} ({\it fully symmetric}) sequence space if
\[
\eta \in E, \  \xi \in l^\ii, \  \xi^* \leq \eta^*\text{\ (resp.},\ \xi^*\prec \prec \eta^*) \ \Longrightarrow \ \xi \in E \text{\ \ and\ \ }\| \xi \|_E \leq \| \eta\|_E.
\]
Every fully symmetric sequence space is a symmetric sequence space. The converse is not true in general. At the same time, any separable symmetric sequence space  is a fully symmetric space.

If $(E,\|\cdot\|_E)$ is a symmetric sequence  space, then 
\[
\|\xi\|_E=\|\,|\xi|\,\|_E = \|\xi^*\|_E \text{ \ \ for all\ \ }\xi\in E.
\]
Besides, if $E_h=\{\{\xi_n\}_{n = 1}^\ii \in E :\ \xi_n\in\mathbb R\text{ \ for each\ }n\}$, where  $\mathbb R$ is the set of real numbers, then $(E_h,\|\cdot\|_E)$ is a Banach lattice with respect to the natural partial order
\[
\{\xi_n\}\leq\{\eta_n\} \ \Longleftrightarrow \ \xi_n \leq \eta_n \ \ \text{for all} \ \ \ n \in \mathbb N.
\]

Immediate examples of fully symmetric sequence spaces are $(l^\ii,\|\cdot\|_\ii)$ and $(c_0,\|\cdot\|_\ii)$ and the Banach lattices
\[
l^p=\left\{\xi=\{\xi_n\}_{n=1}^\ii\in c_0:\ \|\xi\|_p=\left(\sum_{n=1}^\ii|\xi_n|^p\right)^{1/p}<\ii\right\},\ 1\leq p<\ii.
\]
For any symmetric sequence space $(E,\|\cdot\|_{E})$ the following continuous embeddings hold \cite[Ch.\,2, \S\,6, Theorem 6.6]{bs}:
\begin{equation}\label{e1}
(l^1,\|\cdot\|_1)\su (E,\|\cdot\|_E) \subset (l^\ii,\|\cdot\|_\ii).
\end{equation}
Besides, $\|\xi\|_E\leq \|\xi\|_1$ for all $\xi \in l^1$ and  $\|\xi\|_\ii\leq \|\xi\|_E$ for all $\xi \in E$.

If there is $\xi\in E\sm c_0$, then $\xi^*\geq\al\mathbf1$ for some $\al>0$, where $\mathbf1 = \{1,1,...\}$. Consequently, $\mathbf1\in E$ and $E = l^\ii$. Therefore, either $E\su c_0$ or $E=l^\ii$.

Now, let $(\mc H,(\cdot, \cdot))$ be an infinite-dimensional Hilbert space over $\mathbb C$, and let $(\mc B(\mc H),\|\cdot\|_\ii)$ be the $C^*$-algebra of bounded linear operators in $\mc H$.  Denote by $\mc K(\mc H)$
($\mc F(\mc H)$) the two-sided ideal of compact (respectively, finite rank) linear operators in $\mc B(\mc H)$. It is well known that, for any proper two-sided ideal $\mc I\su\mc B(\mc H)$, we have $\mc F(\mc H)\su\mc I$, and if $\mc H$ is separable, then $\mc I\su\mc K(\mc H)$ (see, for example, \cite[Proposition 2.1]{si}).
At the same time, if $\mc H$ is a non-separable Hilbert space, then there exists a proper two-sided ideal $\mc I\su\mc B(\mc H)$ such that $\mc K(\mc H) \subsetneqq \mc I$.

Denote $\mc B_h(\mc H)=\{x \in\mc B(\mc H): x =x^{\ast}\}$, $\mc B_+(\mc H)=\{x \in\mc B(\mc H): x\ge 0\}$, and let
$\tau:\mc B_+(\mc H)\to [0, \ii]$ be the {\it canonical trace} on $\mc B(\mc H)$, that is,
\[
\tau(x)=\sum\limits_{j \in J} (x\varphi_j,\varphi_j), \ \ x\in\mc B(\mc H),
\]
where $\{\varphi_j\}_{j \in J}$ is an orthonormal basis in $\mc H$ (see, for example, \cite[Ch.\,7, E.\,7.5]{sz}).

Let $\mc P(\mc H)$ be the lattice of projections in $\mc H$. If $\mathbf 1$ is the identity of $\mc B(\mc H)$ and
$e\in\mc P(\mc H)$, we will write $e^\perp=\mathbf 1-e$.

Let $x\in\mc B(\mc H)$, and let $\{e_\lb\}_{\lb\ge 0}$ be the spectral family of projections for the absolute value
$|x|= (x^*x)^{1/2}$ of $x$, that is, $e_\lb=\{|x|>\lb\}$.
If $t>0$, then the {\it $t$-th generalized singular number} of $x$, or the {\it non-increasing rearrangement} of $x$,
is defined as
\[
\mu_t(x)=\inf\{\lb>0:\ \tau(e_\lb^\perp)\leq t\}
\]
(see \cite{fk}).

A non-zero linear subspace  $X\su\mc B(\mc H)$ with a Banach norm $\|\cdot\|_X$ is called {\it symmetric} ({\it fully symmetric}) if the conditions
\[
x\in X, \ y\in\mc B(\mc H), \ \mu_t(y)\leq \mu_t(x)\text{ \ \ for all\ \ } t>0 \
\]
(respectively,
\[
x\in X, \ y\in\mc B(\mc H), \ \int\limits_0^s\mu_t(y)dt\leq\int\limits_0^s\mu_t(x)dt \text{ \ \ for all\ \ } s>0 \ \  (\text {writing} \  y \prec\prec x))
\]
imply that $y\in X$ and $\| y\|_X\leq \| x\|_X$.

The spaces $(\mc B(\mc H),\|\cdot\|_\ii)$ and $(\mc K(\mc H),\|\cdot\|_\ii)$ as well as the classical Banach two-sided ideals
\[
L^p(\mc B(\mc H), \tau)=\mc C^p =\{x\in\mc K(\mc H) :\ \|x\|_p=\tau(|x|^p)^{1/p}<\ii\}, \  1 \leq p<\ii,
\]
are examples of fully symmetric spaces.

It should be noted that  for every symmetric space $(X,\|\cdot\|_X) \su \mc B(\mc H)$ and all $x\in X$, $a, b\in\mc B(\mc H)$,
\[
\|x\|_X = \|\,|x|\,\|_X = \|x^*\|_X, \ \ axb \in X,\text{\ \ and\ \ }\|axb\|_X \leq \|a\|_\ii\|b\|_\ii\|x\|_X.
\]

\begin{rem}\label{r1}
If $X\su\mc B(\mc H)$ is a symmetric space and there exists a projection $e\in \mc P(\mc H)\cap X$ such that $\tau(e)=\ii$, that is, $\dim e(\mc H)= \ii$, then $\mu_t(e)=\mu_t(\mathbf 1)= 1$ for every $t \in (0, \infty)$. Consequently,  $\mathbf 1 \in X$ and $X = \mc B(\mc H)$. If $X \neq \mc B(\mc H)$ and $x \in X$, then $e_\lb=\{|x|>\lb\}$ is a finite-dimensional projection, that is,
$\dim e_\lb (\mc H)<\ii$ for all $\lb>0$. This means that $x\in\mc K(\mc H)$, hence $X\su\mc K(\mc H)$. Therefore, either $X=\mc B(\mc H)$ or $X\su\mc K(\mc H)$.

Thus, if $\mc H$ is non-separable, then there exists a proper two-sided ideal  $\mc I\su\mc B(\mc H)$ such that $\mc K(\mc H) \subsetneqq \mc I $ and $(\mc I,\|\cdot\|_\ii)$ is a Banach space which is not a symmetric subspace of $\mc B(\mc H)$.
\end{rem}

If $x\in\mc K(\mc H)$, then $|x|=\sum\limits_{n=1}^{m(x)} s_n(x) p_n$ (if $m(x)=\ii$, the series converges uniformly),
where $\{s_n(x)\}_{n=1}^{m(x)}$ is the set of singular values of $x$, that is, the set of eigenvalues of the compact operator $|x|$ in the decreasing order, and $p_n$ is the projection onto the eigenspace corresponding to $s_n(x)$. Consequently, the non-increasing rearrangement $\mu_t(x)$ of $x\in\mc K(\mc H)$ can be identified with the sequence $\{s_n(x)\}_{n=1}^\ii$, $s_n(x) \downarrow 0$ (if $m(x)<\ii$, we set $s_n(x)=0$ for all $n>m(x)$).

Let $(X,\|\cdot\|_X)\su \mc K(\mc H)$ be a symmetric space. Fix an orthonormal basis  $\{\varphi_j\}_{j \in J}$ in $\mc H$ and choose a countable subset $\{\varphi_{j_n}\}_{n=1}^\ii$.  Let $p_n$ be the one-dimensional projection on the  subspace $\mathbb C\cdot\varphi_{j_n}\su\mc H$.  It is clear that the set
\[
E(X)=\left\{\xi=\{\xi_n\}_{n = 1}^\ii\in c_0: \ x_\xi =\sum\limits_{n=1}^\ii\xi_np_n \in X\right\} \]
(the series converges uniformly),
is a symmetric sequence space with respect to the norm $\|\xi\|_{E(X)} = \| x_\xi\|_X$. Consequently, each symmetric subspace $(X,\|\cdot\|_X)\su\mc K(\mc H)$ uniquely generates a symmetric sequence space $(E(X), \|\cdot\|_{E(X)})\su c_0$. The converse is also true: every symmetric sequence space $(E,\|\cdot\|_E)\su c_0$ uniquely generates a symmetric space $(\mc C_E,\|\cdot\|_{\mc C_E})  \su \mc K(\mc H)$ by the following rule (see, for example, \cite[Ch.\,3, Section 3.5]{lsz}):
\[
\mc C_E =\{x\in\mc K(\mc H):\, \{s_n(x)\}\in E\},\ \ \|x\|_{\mc C_E} = \|\{s_n(x)\}\|_{E}.
\]
In addition,
\[
E(\mc C_E)=E, \ \|\cdot\|_{E(\mc C_E)}=\|\cdot\|_E, \ \mc C_{E(\mc C_E)}=\mc C_E, \ \|\cdot\|_{\mc C_{E(\mc C_E)} }=\|\cdot\|_{\mc C_E}.
\]

We will call the pair $(\mc C_E,\|\cdot\|_{\mc C_E})$ a {\it Banach ideal of compact operators} (cf.\,\cite[Ch.\,III]{gk}). It is known that
$(\mc C^p,\|\cdot\|_p)=(\mc C_{l^p},\|\cdot\|_{\mc C_{l^p}})$ for all $1\leq p<\ii$ and
$(\mc K(\mc H),\|\cdot\|_\ii)= (\mc C_{c_0},\|\cdot\|_{\mc C_{c_0}})$.

Hardy-Littlewood-Polya partial order in the Banach ideal $\mc K(\mc H)$ is defined by
\[
x \prec\prec y, \ x, y \in \mc K(\mc H) \ \Longleftrightarrow \ \{s_n(x)\}\prec\prec \{s_n(y)\}.
\]
We say that a Banach ideal $(\mc C_{E}, \|\cdot\|_{\mc C_{E}})$ is {\it fully symmetric} if conditions  $y\in\mc C_E$,
$x\in\mc K(\mc H)$, $x \prec \prec y$ entail that $x\in\mc C_E$ and $ \| x \|_{\mc C_E}\leq\| y\|_{\mc C_E}$. It is clear that $(\mc C_E,\|\cdot\|_{\mc C_E})$ is a fully symmetric ideal if and only if $(E, \|\cdot\|_E)$ is a fully symmetric sequence space.

Examples of fully symmetric ideals include $(\mc K(\mc H),\|\cdot\|_\ii)$ as well as the Banach ideals $(\mc C^p,\|\cdot\|_p)$  for all $1\leq p<\ii$. It is clear that $\mc C^1\su\mc C_E\su\mc K(\mc H)$ for every symmetric sequence space $E \su c_0$ with $\|x\|_{\mc C_E}\leq\|x\|_1$ and $\|y\|_\ii\leq \|y\|_{\mc C_E}$ for all $x\in\mc C^1$ and $y\in\mc C_E$.

We will need the following property of Hardy-Littlewood-Polya partial order.
\begin{pro}\label{p1}
 If $x, y, y_k\in\mc K(\mc H)$ are such that $y_k\prec\prec x$ for all $k\in\mathbb N$ and $\|y_k-y\|_\ii\to 0$ as $k \to \ii$, then
$y\prec\prec x$.
\end{pro}
\begin{proof}
Since $y_k\prec\prec x$, it follows that $\sum_{n=1}^m s_n(y_k)\leq\sum_{n=1}^m s_n(x)$ for all $m, k\in\mathbb N$. By \cite[Ch.II, \S\,2, Sec.\.3, Corollary 2.3]{gk}, $|s_n(y_k)-s_n(y)|\leq \|y_k-y\|_\ii \to 0$, hence $\sum_{n=1}^m s_n(y_k) \to \sum_{n=1}^m s_n(y)$ as $k \to \ii$ for every $m\in\mathbb N$. Therefore
\[
\sum_{n=1}^m s_n(y)= \lim\limits_{k \to \ii} \sum\limits_{n=1}^m s_n(y_k) \leq \sum\limits_{n=1}^m s_n(x)
\]
for all $m$.
\end{proof}

Define
\[\mc R_\tau=\{x \in\mc B(\mc H):\ \mu_t(x) \to 0 \text{ \ as \ } t\to\ii\}.
\]
By \cite[Proposition 2.7]{ddp}, $\mc R_\tau$ is the closure of $\mc C^1$ in $(\mc B(\mc H), \|\cdot\|_\ii)$, implying that
$\mc R_\tau=\mc K(\mc H)$. Therefore $\lim\limits_{t\to\ii}\mu_t(x)>0$ for every $x\in\mc B(\mc H)\sm\mc K(\mc H)$; in particular,  there exists $\lb >0$ such that $\tau\{|x| > \lb\}=\ii$.

A linear operator $T:\mc B(\mc H)\to\mc B(\mc H)$ is called a {\it Dunford-Schwartz operator} if
\[
\|T(x)\|_1\leq\| x\|_1\text{\ \ for all\ \ }x\in\mc C^1\text{ \ \ and \ \ } \|T(x)\|_\ii\leq\| x\|_\ii\text{\ \ for all\ \ }x\in\mc B(\mc H).
\]
In what follows, we will write $T\in DS$ ($T\in DS_+$) to indicate that $T$ is a Dunford-Schwartz operator (respectively, a positive Dunford-Schwartz operator, that is, $T\in DS$ and $T(\mc B_+(\mc H))\su\mc B_+(\mc H)$).

Any fully symmetric sequence space $E$ is an exact interpolation space in the Banach pair $(l^1, l^\ii)$ (see, for example, \cite [Ch.\,II, \S\,4, Sec.\,2] {kps}). Therefore, for such $E$, the fully symmetric ideal $\mc C_E$ is an exact interpolation space  in the Banach pair $(\mc C^1,\mc B(\mc H))$; see \cite [Theorem 2.4]{DDP}. It then follows that $T(\mc C_E)\su\mc C_E$ and $\|T\|_{\mc C_E\to\mc C_E}\leq 1$ for all $T\in DS$. In particular, $T(\mc K(\mc H))\su\mc K(\mc H)$ and the restriction of $T$ on $\mc K(\mc H)$ is a linear contraction (also denoted by $T$). We note that if $T\in DS$, then
$A_n(T)\in DS$; also, $T(x)\prec\prec x$ and $A_n(T)(x)\prec\prec x$ for any $x\in\mc K(\mc H)$ and $n$.

\section{Individual ergodic theorem  in  fully symmetric ideals of compact operators}
Let $\mc H$, $\tau:\mc B_+(\mc H)\to [0, \ii]$, and $\mc C^1$ be as above. Utilizing Theorem \ref{t11} with $\mc M=\mc B(\mc H)$ and taking into account that $\tau(e)\ge1$ for every $0\neq e\in\mc P(\mc H)$, we arrive at the following.

\begin{teo} \label{t31}
Given $T\in DS_+$ and $x\in \mc C^1$, there exists $\wh x\in\mc C^1$ such that
$\|A_n(T)(x) -\wh x\|_\ii\to 0$ as $n\to\ii$.
\end{teo}

Theorem \ref{t31} can be extended to the fully symmetric ideal $\mc K(\mc H)$. In fact, such an extension holds for any $T\in DS$:

\begin{teo}\label{t32}
Let $T\in DS$ and $x\in\mc K(\mc H)$. Then there exists $\wh x\in\mc K(\mc H)$ such that $\|A_n(T)(x)-\wh x\|_\ii\to 0$ as $n\to\ii$.
\end{teo}
\begin{proof}
Since $T(\mc C^2)\su\mc C^2$, $\|T\|_{\mc C^2 \to \mc C^2} \leq 1$ and the Banach space $\mc C^2$ is reflexive, by the mean ergodic theorem \cite[Ch.\,VIII, \S\,5, Corollary 4]{ds}, the sequence $\{A_n(T)(x)\}$ converges strongly in $\mc C^2$, that is, for every $x\in\mc C^2$ there exists $\wh x\in \mc C^2$ such that $\|A_n(T)(x)-\wh x\|_2\to 0$. As $\|\xi\|_\ii\leq \|\xi\|_2$ for all $\xi\in l^2$, it follows that $\|x\|_\ii \leq \|x\|_2$ for all $x \in\mc C^2$. Consequently,
\[
\|A_n(T)(x)-\wh x\|_\ii\to 0\text{ \ \ for every\ \ }x\in \mc C^2.
\]

Let now $x\in\mc K(\mc H)$ and $\ep>0$. Then there exists $x_\ep\in\mc F(\mc H)\su\mc C^2$ such that
$\|x-x_\ep\|_\ii<\ep/4$. Since the sequence $A_n(T)(x_\ep)$ converges uniformly, there exists $N=N(\ep)$ such that
\[
\|A_m(T)(x_\ep)-A_n(T)(x_\ep)\|_\ii<\frac\ep2\text{ \ \ whenever\ \ }m, n\geq N.
\]
Therefore,
\[
\begin{split}
\|A_m(T)(x)&-A_n(T)(x)\|_\ii \leq \|A_m(T)(x-x_\ep)-A_n(T)(x-x_\ep)\|_\ii\\
&+\|A_m(T)(x_\ep)-A_n(T)(x_\ep)\|_\ii\leq 2\|x-x_\ep\|_\ii+\frac\ep2<\ep.
\end{split}
\]
for all $m, n\ge N$.
Thus, since the space $(\mc K(\mc H),\|\cdot\|_\ii)$ is complete, there exists $\wh x\in\mc K (\mc H)$ such that
$\|A_n(T)(x)-\wh x\|_\ii \to 0$.
\end{proof}

By virtue of Theorem \ref{t32}, we now derive part (i) of Theorem \ref{t12}, an individual ergodic theorem in fully symmetric ideals of compact operators:

\begin{teo}\label{t33}
Let $\mc C_E$ be a fully symmetric ideal of compact operators, and let $T\in DS$.
Then, given $x\in \mc C_E$, the averages $A_n(T)(x)$ converge uniformly to some $\wh x\in\mc C_E$.
\end{teo}
\begin{proof}
As $\mc C_E\su\mc K(\mc H)$, it follows from Theorem \ref{t32} that the sequence
$\{A_n(T)(x)\}$ converges uniformly to some $\wh x\in\mc K(\mc H)$, while Proposition \ref{p1}
implies that $\wh x \prec\prec x$, hence $\wh x\in\mc C_E$.
\end{proof}

The rest of this section is devoted to proving part (ii) of Theorem \ref{t12}: if $x\in\mc B(\mc H)\sm\mc K(\mc H)$, then there exists $T\in DS$ such that the sequence $\{A_n(T)(x)\}$ does not converge uniformly (Theorem \ref{t36} below).

We begin with a Dunford-Schwartz operator acting in the Banach space $(l^\ii,\|\cdot\|_\ii)$, that is, when a linear operator  $T: l^\ii\to l^\ii$ is such that $\|T(\xi)\|_1\leq \| \xi\|_1$ for all $\xi\in l^1$ and $\|T(\xi)\|_\ii\leq\|\xi\|_\ii$ for all $\xi\in l^\ii$ (writing $T\in DS$). In this case, we have a commutative version of Theorem \ref{t12}\,(ii) (cf. \cite[Theorem 3.3]{cl3}):

\begin{teo}\label{t33a}
If $\xi\in l^\ii\sm c_0$, then there exists $T\in DS$  such that the averages $A_n(T)(\xi)$ do not converge coordinate-wise, hence uniformly.
\end{teo}
\begin{proof}
Let $\{\xi_n\}_{n = 1}^\ii\in l^\ii_h\sm c_0$. If $\xi=\xi_+ - \xi_-$, then either $\xi_+=\{(\xi_+)_n\}_{n=1}^\ii\in l^\ii_h\sm c_0$ or $\xi_-\in l^\ii_h\sm c_0$, so let us assume the former. In addition, we may assume that $\lim\limits_{n\to \ii}((\xi_+)^*)_n =1$. It is clear that the set $\{n\in\mathbb N: (\xi_+)_n \geq 1\}$ is infinite. In addition, the set
\[
G = \{m \in \mathbb N : 1 \leq (\xi_+)_m \leq 2\}= \{m_1 < m_2<\dotsc < m_s<\dots \}
\]
is also infinite.

Let $1=n_0,n_1, n_2 \dots$ be an increasing sequence of positive integers. Define the function $\varphi:\mathbb N \to \mathbb R$ by
$$
\varphi(m)=\chi_{\{m_{n_0}\}}+\sum_{k=0}^\ii\left(\chi_{\{m_{n_k+1},m_{n_k+2}, ...,m_{n_{k+1}-1}\}}(m)-
\chi_{\{m_{n_{k+1}}\}}(m)\right) \text{ \ if \ } m\in G;
$$
$$
\varphi(m)=0 \text{ \ if \ } m\notin G.
$$
Then we have
$$
\varphi(m_1)=1, \ \varphi(m_2)=1, \ \ldots, \ \varphi(m_{{n_1-1}})=1, \ \varphi(m_{n_1})=-1,
$$
$$
\varphi(m_{{n_1}+1})=1, \ \varphi(m_{{n_1}+2})=1, \  \ldots, \ \varphi(m_{n_2-1})=1,\ \varphi(m_{n_2})=-1,
$$
$$
\varphi(m_{{n_2}+1})=1, \ \varphi(m_{{n_2}+2})=1, \ \ldots, \ \varphi(m_{n_3-1})=1,\ \varphi(m_{n_3})=-1, \ \ \ldots
$$
Let $\pi:\mathbb N\to\mathbb N$ be given by
$$
\pi(m_i) = m_{i+1}\text{ \ if \ }m_i\in G\text{ \ \ and \ \ }\pi(m)=m \text{ \ if \ } m\notin G.
$$
Define a linear operator $T:l^\ii\to l^\ii$ by
$$
T(\{\eta_n\}_{n=1}^\ii)=\{\varphi(n)\,\eta_{\pi(n)}\}_{n=1}^\ii, \ \ \{\eta_n\}_{n=1}^\ii\in l^\ii.
$$
Then, clearly, $ T\in DS$.

Since
$$
T^k(\xi_+)_m=\varphi(m)\varphi(\pi(m))\varphi(\pi^2(m))\ldots \varphi(\pi^{k-1}(m))\,(\xi_+)_{\pi^k(m)}
$$
for all $k, m\in\mathbb N$,  it follows that
\[
\begin{split}
A_{n_1-1}(T)(\xi_+)_{m_1}&=\frac 1{n_1}\sum\limits_{k=0}^{n_1-1}T^k(\xi_+)_{m_1}\\
&=\frac1{n_1}\bigg((\xi_+)_{m_1}+\sum\limits_{k=1}^{n_1-1} \varphi(m_1)\varphi(m_2)\ldots
\varphi(m_k)\,(\xi_+)_{m_{k+1}}\bigg)\\
&=\frac 1 {n_1} \sum\limits_{k=0}^{n_1-1} (\xi_+)_{m_{k+1}}\ge 1> \frac 12.
\end{split}
\]
Further, since
\[
\begin{split}
A_{n_2-1}(T)(\xi_+)_{m_1}&=\frac 1{n_2}\left(\sum\limits_{k=0}^{n_1-1}(\xi_+)_{m_{k+1}}-\sum\limits_{k=n_1}^{n_2-1}(\xi_+)_{m_{k+1}}\right)\\
&\leq \frac 1{n_2-1} ( 2 n_1-(n_2-n_1-1)),
\end{split}
\]
there exists such $n_2>n_1$ that
\[
A_{n_2-1}(T)(\xi_+)_{m_1}<-\frac12.
\]
As
\[
A_{n_3-1}(T)(\xi_+)_{m_1}=\frac 1{n_3}\left(\sum\limits_{k=0}^{n_1-1}(\xi_+)_{m_{k+1}}-\sum\limits_{k=n_1}^{n_2-1}(\xi_+)_{m_{k+1}}+\sum\limits_{k=n_2}^{n_3-1}(\xi_+)_{m_{k+1}}\right),
\]
one can find $n_3>n_2$ for which
\[
A_{n_3-1}(T)(\xi_+)_{m_1}>\frac12.
\]
Continuing this procedure, we choose $n_1<n_2<n_3<\dots$ to satisfy the inequalities
\[
A_{n_{2k-1}-1}(T)(\xi_+)_{m_1}>\frac 12 \text{ \ \ and \ \ } A_{n_{2k}-1}(T)(\xi_+)_{m_1}<-\frac 12, \ \ k=1,2,\dots,
\]
implying that $\{A_n(T)(\xi_+)_{m_1}\}$ is a divergent sequence.

Finally, note that $T(\xi_-)=0$, which implies that
\[
 A_n(T)(\xi)= A_n(T)(\xi_+)-A_n(T)(\xi_-)=A_n(T)(\xi_+)-\frac1{n+1}\xi_-,
\]
so, the sequence $\{A_n(T)(\xi)\}$ does not converge coordinate-wise, hence uniformly.

If $\xi\in l^\ii\sm c_0$, then
\[
\xi=\operatorname{Re}\xi +i \operatorname{Im}\xi,\text{ \ where\ }\operatorname{Re}\xi=\frac{\xi + \overline{\xi}}{2},\ \operatorname{Im}\xi = \frac{\xi- \overline{\xi}}{2i} \in l_h^\ii.
\]
As shown above, there exists $T \in DS$ such that the sequence $\{ A_n(T)(\operatorname{Re}\xi)\}$ does not converge coordinate-wise. Then the sequence $\{ A_n(T)(\xi ) \}$ also does not converge coordinate-wise, hence uniformly.
\end{proof}

Now we need a statement on the existence of {\it conditional expectation} in a von Neumann algebra $\mc B(\mc H)$ (see, for example, \cite{ta}).

\begin{teo}\label{t35}
Let $\mc N$ be a  von Neumann subalgebra in $\mc B(\mc H)$ such that the restriction of the trace $\tau$ on $\mc N$  is a semifinite trace. Then there exists a unique linear map  $U: \mc B(\mc H)\to\mc N$ (conditional expectation on $\mc N$), having the following properties:
\begin{enumerate}[(i)]
\item
$\tau(x)=\tau(U(x))$ for all $x\in\mc C^1$;
\item
$U(x)=x$ for all $x \in \mc N$;
\item
$U \in DS_+$; moreover, $\|U\|_{\mc B(\mc H) \to \mc B(\mc H)} = 1$ and $\|U\|_{\mc C^1 \to \mc C^1} = 1$.
\end{enumerate}
\end{teo}

Assume first that $(\mc H, (\cdot, \cdot))$  is a separable  infinite-dimensional  complex Hilbert space.  Fix an orthonormal basis $\{\varphi_n\}_{n\in \mathbb N}$ in  $\mc H$.  Let $p_n$ be the one-dimensional projection on the linear subspace $\mathbb C\cdot\varphi_n\su\mc H$. It is clear that $p_m p_n = 0$ for all $m,n\in\mathbb N$, $n \neq m$.

For any $\xi=\{ \xi_n \}_{n = 1}^\ii\in  l^\ii$ and $h=\sum\limits_{n=1}^\ii(h,\varphi_n)\varphi_n\in\mc H$ we set
$$
x_\xi(h) = \sum\limits_{n=1}^\ii \xi_n(h,\varphi_n)\varphi_n = \sum\limits_{n=1}^\ii \xi_n p_n (h).
$$
It is clear that $x_\xi\in\mc B(\mc H)$ and $x_\xi= (wo)-\sum\limits_{n=1}^\ii\xi_n p_n$, where  $(wo)$ stands for the weak operator topology. In addition,
\[
\mc N=\{ x_\xi\in\mc B(\mc H):\ \xi=\{ \xi_n \}_{n = 1}^\ii\in l^\ii\}
\]
is the smallest commutative von Neumann subalgebra  in $\mc B(\mc H)$ containing all projections $p_n$. Besides, the restriction of the trace $\tau$ on $\mc{N}$  is a semifinite trace.

Define the linear map $\Phi: (\mc N,\|\cdot\|_\ii)\to ( l^\ii,\|\cdot\|_\ii)$ by setting $\Phi(x_\xi)=\xi$.

\begin{pro}\label{p33}
$\Phi$ is a positive linear surjective isometry.
\end{pro}
\begin{proof}
By definition of $\Phi$, we have $\Phi(\mc N) = l^\ii$.
Using \cite[Ch.\,1, \S\,1.1, E.\,1.1.11]{lsz}, we see that
\[
\|x_\xi\|_\ii=\|\xi\|_\ii=\|\Phi(x_\xi)\|_\ii,
\]
that is, $\Phi$ is a linear surjective isometry.

Since $\xi=\{ \xi_n \}_{n = 1}^\ii\ge 0$ whenever $x_\xi \in \mc N_+$, the map $\Phi$ is positive.
\end{proof}

Let $(E,\|\cdot\|_E)\su c_0$ be a symmetric sequence space, and let $\mc N_E=\mc N\cap\mc C_E$. If $x_\xi= \sum\limits_{n=1}^\ii \xi_n p_n\in\mc N_E$, then $\{s_n(x_\xi)\}_{n=1}^\ii=\{ \xi_n^* \}\in E$, hence $\{\xi_n \}\in E$. In addition,
\[
\|x_\xi\|_{\mc C_E}=\|\{ \xi_n^* \}\|_E=\|\{ \xi_n \}\|_E.
\]
Therefore, we have the following.
\begin{pro}\label{p34}
If $(E,\|\cdot\|_E)\su c_0$ is a symmetric sequence space, then the restriction $\Phi|_{\mc N_E}: (\mc{N}_E,\|\cdot\|_{\mc C_E}) \to (E,\|\cdot\|_E)$ is a positive linear surjective isometry (we denote this restriction also by $\Phi$).
\end{pro}

\begin{teo}\label{t36}
If $x\in\mc B(\mc H)\sm\mc K(\mc H)$, then there exists $T \in DS$ such that the sequence $\{A_n(T)(x)\}$ does not converge uniformly.
\end{teo}
\begin{proof}
Assume first that $x\ge 0$ and $\mc H$ is separable. Since $x \notin \mc K(\mc H)$, it follows that there exists a spectral projection $e_\lb=\{x >\lb\}$, $\lb>0$, such that $\tau(e_\lb)=\ii$. Choose an orthonormal basis $\{\varphi_n\}_{n=1}^\ii$ in $\mc H$ such that $e_\lb\ge p_{n_i}$ for some sequence $\{n_i\}_{i=1}^\ii$, where $p_n$ is the one-dimensional projection on the subspace $\mathbb C\cdot\varphi_n\su\mc H$.

Let $\mc N=\{x_\xi\in\mc B(\mc H):\  \xi=\{ \xi_n \}_{n = 1}^\ii\in  l_\ii\}$ be the smallest commutative von Neumann subalgebra in $\mc B(\mc H)$ containing all projections $p_n$. By virtue of Theorem \ref{t35}, there exists a conditional expectation $U: \mc B(\mc H)\to\mc N$ such that
\[
0\leq y=U(x)\ge U(\lb\, e_\lb)\ge\lb\, U(p_{n_i})=\lb p_{n_i}\text{ \ \ for all\ \ } i\in\mathbb N.
\]
Consequently, $y\notin\mc R_\tau$ and $y = x_\xi \in \mc N$, where $0\leq \xi=\{ \xi_n \}_{n = 1}^\ii\in l^\ii\sm c_0$. Besides, by definition of $\Phi$, we have $\Phi(y)=\xi$.

Next, by Theorem \ref{t33a}, there exists an operator $S: l^\ii\to l^\ii$, $S\in DS$, such that the sequence $\{A_n(S)(\xi)\}$ does not converge uniformly. Consider the operator
\[
T=\Phi^{-1} S \Phi U:\, \mc B(\mc H)\to\mc N\su\mc B(\mc H).
\]
It is clear that $T\in DS$. As $y=U(x)\in \mc N$, hence $U(y)=y$ (see Theorem \ref{t35}\,(ii)), and
$U\Phi^{-1}=\Phi^{-1}$, we have $T^k(y)=\Phi^{-1} S^k \Phi(y)$ for each $k\in\mathbb N$.

Since $\Phi^{-1}$ is an isometry and
\[
A_n(T)(y)=\frac1{n+1}\sum_{k=0}^nT^k(y)=\Phi^{-1}\left(\frac1{n+1}\sum_{k=0}^nS^k\Phi(y)\right)
=\Phi^{-1}(A_n(S)(\xi)),
\]
for all $n \in \mathbb{N}$, it follows that the sequence $\{A_n(T)(y)\}_{n=1}^\ii$ does not converge uniformly.

Now, as above, $y=U(x)\in\mc N$ entails $T^k(x)=\Phi^{-1}S^k\Phi(y)=T^k(y)$ for all $k\in \mathbb N$. Therefore, we have
\[
A_n(T)(x)-A_n(T)(y)=\frac1{n+1}(x-y),
\]
and it follows that the sequence $\{A_n(T)(x)\}_{n=1}^\ii$ also does not converge uniformly.

Let now $\mc H$ be non-separable, and let $0\leq x\in\mc B(\mc H)\sm\mc K(\mc H)$. Since $x \notin \mc K(\mc H)$ it follows that there exists a spectral projection $e_{\lb}=\{x >\lb\}$, $\lb> 0$, such that $\tau(e_\lb)=\ii$. Choose an orthonormal basis $\{\varphi_j\}_{j \in J}$  in  $\mc H$ such that $e_{\lb} \geq p_{j_n}$ for some sequence $\{j_n\}_{n=1}^\ii$, where $p_j$ is the one-dimensional projection on the subspace $\mathbb C\cdot\varphi_j\su\mc H$.
If $p=\sup\limits_{n\in\mathbb N}p_{j_n}$, then $\mc H_0=p(\mc H)$ is a separable infinite-dimensional Hilbert subspace in $\mc H$ such that $\mc K(\mc H_0)=p\mc K(\mc H)p$.

Since $z=pxp\in\mc B_+(\mc H_0)$ and $z\ge\lb pe_\lb p\ge\lb p$, it follows that  $z\in\mc B_+(\mc H_0)\sm\mc K(\mc H_0)$. In view of the above, there exists a Dunford-Schwartz operator $D_0:\mc B(\mc H_0)\to\mc B(\mc H_0)$ such that the sequence $\{A_n(D_0)(z)\}_{n=1}^\ii$  does not converge uniformly.

It is clear that $D(y)=D_0(pyp)$, $y\in\mc B(\mc H)$, is a Dunford-Schwartz operator in $\mc B(\mc H)$ such that $D^k(x) = D_0^k(z)$ for each $k\in\mathbb N$. Then
\[
A_n(D)(x)-A_n(D_0)(z)=\frac1{n+1} (x-z),
\]
and we conclude that the sequence $\{A_n(D)(x)\}_{n=1}^\ii$ does not converge uniformly.

Further, let $x\in\mc B(\mc H)_h\sm\mc K(\mc H)$. Then $x = x_+ - x_-$ such that $x_+, x_- \in\mc B_+(\mc H)$ and $x_+\, x_- =0$. It is clear that either $x_+\in\mc B_+(\mc H)\sm\mc K(\mc H)$ or $x_-\in\mc B_+(\mc H)\sm\mc K(\mc H)$. Suppose that $x_+\in\mc B_+(\mc H)\sm \mc K(\mc H)$ and let $q=\mathbf s(x_+)$ be the support of $x_+$. If $\mc L= q(\mc H)$, then, by the above, there exists a Dunford-Schwartz operator $S_0:\mc B(\mc L)\to \mc B(\mc L)$ such that the sequence $\{A_n(S_0)(x_+)\}_{n=1}^\ii$ does not converge uniformly. Consider the operator $S:\mc B(\mc H)\to\mc B(\mc H)$ given  by $S(y)=S_0(qyq)$, $y\in\mc B(\mc H)$. It is clear that $S \in DS$, $S(x)=S_0(qxq)= S_0(x_+)$, and $S^k(x)= S_0^k (x_+)$ for all $k\in\mathbb N$. Consequently, the sequence $\{A_n(S)(x)\}_{n=1}^\ii$ does not converge uniformly.

Finally, if $x\in\mc B(\mc H)\sm\mc K(\mc H)$ is arbitrary, then, repeating the ending of the proof of Theorem \ref{t33a}, we obtain that there exists $T\in DS$ such that the sequence $\{A_n(T)(x)\}$  does not converge uniformly.
\end{proof}

Let $X\su\mc B(\mc H)$ be a fully symmetric space. We will write $X\in (IET)$ if $X$ satisfies the following individual ergodic theorem: for any $x\in X$ and $T\in DS$ there exists $\wh x\in X$ such that  $\|A_n(T)(x)-\wh x\|_\ii \to 0$ as $n\to\ii$.
Theorems  \ref{t33} and \ref{t36} yield the following criterion.

\begin{teo}\label{t37} 
Let $X\su\mc B(\mc H)$ be a fully symmetric space. Then the following conditions are equivalent:
\begin{enumerate}[(i)]
\item
$X\in (IET)$;
\item
$X\su\mc K(\mc H)$.
\end{enumerate}
\end{teo}

\section{Mean ergodic theorem  in  fully symmetric ideals of compact operators}

In this section, our goal is to prove Theorem \ref{t14}. So, let $(E,\|\cdot\|_E)\su c_0$ be a fully symmetric sequence space, and let $(\mc C_E,\|\cdot\|_{\mc C_E})$ be a fully symmetric ideal generated by $(E,\|\cdot\|_E)$. We will write $\mc C_E \in (MET)$ if the ideal $(\mc C_E,\|\cdot\|_{\mc C_E})$ satisfies the mean ergodic theorem, that is, if for any $x\in\mc C_E$ and $T\in DS$ there exists $\wh x\in\mc C_E$ such that  $\|A_n(T)(x)-\wh x\|_{\mc C_E} \to 0$ as $n \to \ii$,

\begin{pro}\label{p41}
$\mc C^1 \notin (MET)$.
\end{pro}
\begin{proof} Let $S: l^\ii \to l^\ii$ be the positive  Dunford-Schwartz operator defined by
\[
S(\{\xi_n\}_{n = 1}^\ii) = \{0,\xi_1,\xi_2, \dots \}, \ \ \{\xi_n\}_{n = 1}^\ii\in l^\ii.
\]
If $\xi=\{1,0,0,\dots \} \in l^1$, then
\[
\begin{split}
&\ \ \ \ \|A_{2n-1}(S)(\xi) - A_{n-1}(S)(\xi) \|_1\\
&=\left\|\frac1{2n}\{\underbrace{1,1,\dots,1}_{2n},0,0,\dots\}-\frac1{n}\{\underbrace{1,1,\dots,1}_n,0,0,\dots\}
\right\|_1=1.
\end{split}
\]
Consequently, the sequence $\{A_n(S)(\xi)\}$ does not converge in the norm $\|\cdot\|_1$.

Let  $\{\varphi_j\}_{j \in J}$ be an orthonormal basis  in the Hilbert space $\mc H$, and let $\{\varphi_{j_n}\}_{n=1}^\ii$ be a countable subset of $ \{\varphi_j\}_{j \in J}$.  Let $p_n$ be the one-dimensional projection on the subspace $\mathbb C \cdot \varphi_{j_n}\su\mc H$, and let $p=\sup\limits_{n \in\mathbb N}p_n$. It is clear that $\mc H_0=p(\mc H)$ is a separable infinite-dimensional Hilbert subspace in $\mc H$ and $\mc K(\mc H_0) = p\mc K(\mc H)p$. Let
\[
\mc N(\mc H_0)=\left\{x_\xi = (wo)-\sum\limits_{n=1}^\ii \xi_n p_n\in\mc B(\mc H_0):\ \xi=\{ \xi_n \}_{n = 1}^\ii\in l^\ii\right\}
\]
be the smallest commutative von Neumann subalgebra in $\mc B(\mc H_0)$ containing the projections $p_n$, $n\in\mathbb N$, and let $\Phi(x_\xi)=\{\xi_n\}_{n=1}^\ii$ be the positive linear surjective isometry from $(\mc N(\mc H_0),\|\cdot\|_\ii)$ onto $( l_\ii,\|\cdot\|_\ii)$ given in Proposition \ref{p33}. Finally, let $U:\mc B(\mc H_0)\to\mc N(\mc H_0)$ be the conditional expectation given in Theorem \ref{t35}.

It is clear that
$$
T=\Phi^{-1} S\Phi U:\mc B(\mc H_0)\to\mc N(\mc H_0)\su\mc{B}(\mc H_0)
$$
is a positive Dunford-Schwartz operator. If  $\xi=\{1,0,0,\dots \} \in l^1$ and $x_\xi=\Phi^{-1}(\xi)$, then  $x_\xi \in \mc{N}(\mc H_0) \cap \mc C^1$ (see Proposition \ref{p34}), and  $U(x_\xi)=x_\xi$ (see Theorem \ref{t35} (ii)). Consequently,
$$
T(x_\xi)=\Phi^{-1} S\Phi U(x_\xi)=\Phi^{-1}S\Phi(x_\xi).
$$
Now, repeating the proof of Theorem \ref{t36}, we conclude that the averages $\{A_n(T)(x_\xi)\}$ do not converge in the norm $\|\cdot\|_1$, that is, $\mc C^1 \notin (MET)$.
\end{proof}

Here is another sufficient condition for $\mc C_E \notin (MET)$:
\begin{pro}\label{p42}
If $(E,\|\cdot\|_E)\su c_0$ is non-separable fully symmetric sequence space, then $\mc C_E \notin (MET)$.
\end{pro}
\begin{proof}
If $(E,\|\cdot\|_E)\su c_0$ is a non-separable fully symmetric sequence space, then there exists $\xi=\{\xi_n\}_{n = 1}^\ii=\{ \xi^*_n \}_{n = 1}^\ii\in E$, hence $\xi_n\downarrow 0$, such that
\begin{equation}\label{e2}
\|\{\underbrace{0,0,\dots,0}_{n+1},\xi_{n+2},\dots\}\|_E\downarrow\alpha >0.
\end{equation}
Let the operator $S \in DS$ be defined as in the proof of Proposition \ref{p41}. Then $S^k(\xi)=\{\underbrace{0,0,\dots,0}_k,\xi_1,\xi_2,\dots\}$ and
$$
\sum_{k=0}^nS^k(\xi)=\{ \eta_m^{(n)}\}_{m=1}^\ii,
$$
where
$$
\eta_m^{(n)}=\xi_1+\xi_2+\ldots+\xi_m\text{ \ \ for \ \ } 1\leq m\leq n+1
$$
and
$$
\eta_m^{(n)}=\xi_{m-n}+\xi_{m-n+1}+\ldots+ \xi_{m}\text{ \ \ for \ \ }m>n+1.
$$
Since $\xi_n\downarrow  0$, given $1\leq m\leq n+1$, we have
$$
0\leq\frac1{n+1}\,\eta_m^{(n)}\leq\frac1{n+1}\sum_{k=1}^{n+1}\xi_k \to 0\text{ \ \ as\ \ }n\to\ii,
$$
implying that $A_n(S)(\xi)\to 0$ coordinate-wise.

Assume that there exists $\widehat\xi\in E$ such that $\|A_n(S)(\xi)- \widehat{\xi}\|_E\rightarrow 0$. Then we have
$\|A_n(S)(\xi)-\wh\xi\|_\ii\rightarrow 0$; in particular, $A_n(S)(\xi)\to 0$ coordinate-wise, hence $\widehat\xi=0$.

On the other hand, as $\xi_n\downarrow 0$, we obtain
\begin{equation*}
\begin{split}
A_n(S)(\xi)=\bigg\{&\frac{\xi_1}{n+1},\,\frac{\xi_1+\xi_2}{n+1},\dots,\,\frac{\xi_1+\xi_2+\ldots+\xi_{n+1}}{n+1},\,\frac{\xi_2+\xi_3+\ldots+\xi_{n+2}}{n+1},\\
&\frac{\xi_3+\xi_4+\ldots+\xi_{n+3}}{n+1},\dots,\,\frac{\xi_{m-n}+\xi_{m-n+1}+\ldots+ \xi_m}{n+1},\dots\bigg\}\\
&\ge\{\underbrace{0,0,\dots,0}_{n+1},\xi_{n+2},\dots\}.
\end{split}
\end{equation*}
Therefore, in view of (\ref{e2}),
$$
\|A_n(S)(\xi)\|_E\ge\al,
$$
implying that the sequence $\{A_n(S)(\xi)\}$ does not converge in the norm $\|\cdot\|_E$.

Now, if we define the Dunford-Schwartz operator $T \in DS$ as in the proof of Proposition \ref{p41}, then repeating its proof for $x=\Phi^{-1}(\xi)$, we conclude that the sequence $\{A_n(T)(x)\}$ does not converge in $(\mc C_E,\|\cdot\|_{\mc C_E})$. This means that $\mc C_E \notin (MET)$.
\end{proof}

Fix $T \in DS$. By Theorem \ref{t32}, for every $x\in\mc K(\mc H)$ there exists $\wh x\in\mc K(\mc H)$ such that $\|A_n(T)(x)-\wh x\|_\ii \to 0$ as $n \to \ii$. Therefore, one can define a  linear operator $P_T:\mc K(\mc H) \to \mc K(\mc H)$  by setting $P_T(x)=\wh x$. Then we have
$$
\|P_T(x)\|_\ii=\lim\limits_{n\to\ii}\|A_n(T)(x)\|_\ii\leq \|x\|_\ii,
$$
Besides, since the unit ball in $(\mc C^1,\|\cdot\|_1)$ is closed in measure topology \cite[Proposition 3.3]{ddp} and $\|A_n(T)(x)\|_1 \leq \|x\|_1$  for all $x \in \mc C^1$, it follows that $\| P_T(x)\|_1 \leq \|x\|_1$, $x\in\mc C^1$.  Consequently, $\| P_T\|_{\mc C^1\rightarrow \mc C^1}\leq 1$, and, according to \cite[Proposition 1.1]{cl1}, there exists a unique operator $\wh P\in DS$ such that $\wh P(x)=P_T(x)$ whenever $x\in\mc K(\mc H)$. In what follows, we denote $\wh P$ by $P_T$.

\begin{lm}\label{l41}
If $T \in DS$ and $x\in\mc K(\mc H)$, then
\[
P_TT(x)=P_T(x)=TP_T(x).
\]
\end{lm}
\begin{proof}
We have
\[
\|(I-T)A_n(T)(x)\|_\ii=\left\|\frac{(I-T^{n+1})(x)}{n+1}\right\|_\ii\longrightarrow 0 \text{ \ \ as\ \ } n\to\ii.
\]
On the other hand,
\[
TA_n(T)(x) =  \frac1{n+1}\sum_{k=0}^{n} T^k(Tx) \stackrel{\|\cdot\|_\ii}{\longrightarrow} P_T(T(x)),
\]
implying that
\[
(I-T)A_n(T)(x) = A_n(T)(x)-TA_n(T)(x) \stackrel{\|\cdot\|_\ii}{\longrightarrow} P_T(x)-P_TT(x),
\]
hence $P_TT(x)=P_T(x)$.

Now, as $ \|A_n(T)(x)-P_T(x)\|_\ii \to 0$, we have $\|T(A_n(T)(x)) -T(P_T(x))\|_\ii\to 0$ as $n\to\ii$, and the result follows.
\end{proof}

\begin{cor}\label{c41}
If \ $T \in DS$ and $x\in\mc K(\mc H)$, then
\[
T^k(P_T(x))=P_T(x)\text{ \ \ and \ \ } P_T^2(x)=P_T(x).
\]
\end{cor}

We need the following property of separable symmetric sequence spaces \cite[Proposition 2.2]{dds}.
\begin{pro}\label{p43}
Let $(E,\|\cdot\|_E)$ be a separable symmetric sequence space. If $\mc C_E\ni y_n\prec\prec x\in\mc C_E$ for every
$n\in\mathbb N$ and $\|y_n\|_\ii \to 0$ as $n\to\ii$, then $\|y_n\|_{\mc C_E} \to 0$ as $n\to \ii$.
\end{pro}

Now we can finalize the proof of Theorem \ref{t14}:
\begin{proof}
(i) $\Rightarrow$ (ii): Proposition \ref{p42} implies that $E$ is separable. If $E = l^1$ as sets, then the norms
$\|\cdot\|_E$ and $\|\cdot\|_1$ are equivalent  \cite[Part II, Ch.\,6, \S\,6.1]{rgmp}. Therefore, in view of Proposition \ref{p41}, we would have $(\mc C_E,\|\cdot\|_{\mc C_E}) \notin (MET)$, a contradiction.

(ii) $\Rightarrow$ (i): Let $(E,\|\cdot\|_E)$ be separable, $E \neq l^1$, and let $T \in DS$.
If $x\in\mc C_E $ and $y= x-P_T(x)$, then $P_T(y) = 0$, which, by Theorem \ref{t33}, implies $\|A_n(T)(y)\|_\ii\to 0$. Since $E$ is a separable symmetric sequence space, it follows from Proposition \ref{p43} that  \begin{equation}\label{e3}
\|A_n(T)(y)\|_{\mc C_E}\to 0.
\end{equation}
Since $P_T(z)\prec\prec z$ for all $z\in\mc K(\mc H)$ (see Section 2) and $T\in DS$, it follows that $A_n(T)(P_T(x))\prec \prec P_T(x) \prec \prec x$, hence $A_n(T)(P_T(x))-P_T(x)\prec\prec 2x$. Next, as $A_n(T)(P_T(x))\stackrel{\|\cdot\|_\ii}{\longrightarrow}P_T(x)$, Proposition \ref{p43} entails
\begin{equation}\label{e3a}
\|A_n(T)(P_T(x))-P_T(x)\|_{\mc C_E}\to 0.
\end{equation}
Now, utilizing (\ref{e3}) and (\ref{e3a}), we obtain
\begin{equation*}
\begin{split}
\|A_n(T)(x) - P_T(x)\|_{\mc C_E}&=\|A_n(T)(x) - A_n(T)(P_T(x))+A_n(T)(P_T(x))-P_T(x)\|_{\mc C_E}\\
&\leq \|A_n(T)(y)\|_{\mc C_E}+\|A_n(T)(P_T(x))-P_T(x)\|_{\mc C_E}\to 0
\end{split}
\end{equation*}
as $n \to \ii$. Therefore $\mc C_E \in (MET)$.
\end{proof}

\section{Ergodic theorems in Orlicz, Lorentz, and Marcinkiewicz ideals of compact operstors}

In this section we present applications of Theorems \ref{t12} and \ref{t14} to Orlicz, Lorentz and Marcinkiewicz ideals of compact operators.

1. Let $\Phi$ be an {\it Orlicz function}, that is, $\Phi:[0,\ii)\to [0,\ii)$ is left-continuous, convex, increasing and such that $\Phi(0)=0$ and $\Phi(u)>0$ for some $u\ne 0$ (see, for example, \cite[Ch.\,2, \S\,2.1]{es}, \cite[Ch.\,4]{lt}).  Let
\[
 l^\Phi(\mathbb N)=\left \{\xi= \{ \xi_n \}_{n = 1}^\ii\in  l^\ii: \  \sum\limits_{n=1}^\ii \Phi\left (\frac {|\xi_n |}a \right )
<\ii \text { \ for some \ } a>0 \right \}
\]
be the corresponding {\it Orlicz sequence space},  and let
$$\| \xi\|_\Phi=\inf \left \{ a>0:\ \sum\limits_{n=1}^\ii \Phi\left (\frac {|\xi_n |}a \right ) \leq 1\right\}
$$
be the {\it Luxemburg norm} in $ l^\Phi(\mathbb N)$. It is well-known that  $( l^\Phi(\mathbb N), \| \cdot\|_\Phi)$ is a fully symmetric sequence space.

If $\Phi(u)>0$ for all $u\ne 0$, then $\sum\limits_{n=1}^\ii\Phi(a^{-1})=\ii$ for each $a>0$, hence $\mathbf 1=\{1,1,...\}\notin l^\Phi(\mathbb N)$ and $l^\Phi(\mathbb N)\su c_0$. If $\Phi(u)=0$ for all $0\leq u<u_0$, then $\mathbf 1 \in l^\Phi$ and $l^\Phi(\mathbb N)= l^\ii$.

It is said that an Orlicz function $\Phi$ satisfies {\it $(\Delta_2)$-condition at $0$}  if there exist $u_0\in (0,\ii)$ and $k >0 $  such that $\Phi(2u)<k\, \Phi(u)$ for all $0<u< u_0$. It is well known that an Orlicz function $\Phi$ satisfies $(\Delta_2)$-condition at $0$  if and only if $(l^\Phi(\mathbb N), \| \cdot\|_\Phi)$ is separable; see \cite[Ch.\,2, \S\,2.1, Theorem 2.1.17]{es}, \cite[Ch.\,4, Proposition 4.a.4]{lt}.

We also note that  $l^\Phi(\mathbb N)=l^1$, as sets, if and only if $\limsup\limits_{u\to 0}\frac{\Phi(u)}u>0$; see \cite[Ch.\,4, Proposition 4.a.5]{lt}, \cite[Ch.\,16, \S\,16.2]{rgmp}.

If  $\Phi(u)>0$ for all $u\ne 0$, then $l^\Phi(\mathbb N)\su c_0$, and we can define
\[
\mc C^\Phi=\mc C_{l^\Phi(\mathbb N)},\ \ \| x\|_\Phi=\| x\|_{\mc C_{l^\Phi(\mathbb N)}}, \ x \in \mc C^\Phi.
\]
Now, Theorems  \ref{t12} and \ref{t14} yield the following.

\begin{teo}\label{t51} Let $\Phi$ be an Orlicz function such that $\Phi(u)>0$ for all $u\ne 0$. Then
\begin{enumerate}[(i)]
\item
$\mc C^\Phi\in (IET)$;
\item
$(\mc C^\Phi,\| \cdot\|_\Phi)\in (MET)$ if and only if
$\Phi$ satisfies $(\Delta_2)$-condition at $0$  and $\lim\limits_{u\to 0}\frac{\Phi(u)}u=0$.
\end{enumerate}
\end{teo}

2. Let $\psi$ be a concave function on $[0, \ii)$ with $\psi(0)=0$ and $\psi(t)>0$ for all $t>0$, and let
\[
\Lambda_\psi(\mathbb N)=\left \{\xi= \{ \xi_n \}_{n = 1}^{\ii} \in  l^\ii: \  \|\xi \|_{\psi} =
\sum\limits_{n=1}^\ii \xi^*_n(\psi(n)-\psi(n-1))<\ii\right \},
\]
the corresponding {\it Lorentz sequence space}. The pair $(\Lambda_\psi(\mathbb N), \|\cdot\|_\psi)$
is a fully symmetric sequence space; see, for example, \cite[Ch.\,II, \S\,5]{kps}, \cite[Part III, Ch.\,9, \S\,9.1]{rgmp}. Besides, if $\psi(\ii)=\ii$, then $\mathbf 1\notin\Lambda_\psi(\mathbb N)$ and   $\Lambda_\psi(\mathbb N)\su c_0$; if
$\psi(\ii)<\ii$, then $\mathbf 1\in\Lambda_\psi(\mathbb N)$ and $\Lambda_\psi(\mathbb N)= l^\ii$.

It is well known that $(\Lambda_\psi(\mathbb N),\| \cdot \|_{\psi})$ is separable if and only if $\psi(+0) = 0$ and $\psi(\ii) = \ii$; see, for example, \cite[Ch.\,II, \S\,5, Lemma 5.1]{kps}, \cite[Ch.\,9, \S\,9.3, Theorem 9.3.1]{rgmp}. It is clear that
$\lim\limits_{t\to\ii}\frac{\psi(t)}t>0$ if and only if the norms $\|\cdot\|_{\psi}$ and $\|\cdot\|_1$ are equivalent on
$\Lambda_\psi(\mathbb N)$, that is, if $\Lambda_\psi(\mathbb N)=l^1$, as sets.

If $\psi(\ii)=\ii$, then  $\Lambda_\psi(\mathbb N)\su c_0$, and we can define
\[
\mc C_\psi=\mc C_{\Lambda_\psi(\mathbb N)}, \ \ \| x\|_\psi=\| x\|_{\mc{C}_{\Lambda_\psi(\mathbb N)}}, \ x\in\mc C_\psi.
\]
Theorems \ref{t12} and \ref{t14} imply the following.
\begin{teo}\label{t52}
Let $\psi$ be a concave function on $[0, \ii)$ with $\psi(0)=0$, $\psi(t)>0$ for all $t > 0$, and let $\psi(\ii)=\ii$. Then
\begin{enumerate}[(i)]
\item
$C_\psi\in  (IET)$;
\item
$(\mc C_\psi,\| \cdot\|_\psi) \in  (MET)$ if and only if $\psi(+0)=0$  and  $\lim\limits_{t\to\ii}\frac{\psi(t)}t=0$.
\end{enumerate}
\end{teo}

3. Let $\psi$ be as above, and let
\[
M_\psi(\mathbb N)=\left\{\xi=\{ \xi_n \}_{n = 1}^\ii\in l^\ii: \  \|\xi \|_{M_\psi} =
\sup\limits_{n\geq1} \frac1{\psi(n)}\sum\limits_{k=1}^n\xi^*_k<\ii\right\},
\]
the corresponding {\it Marcinkiewicz sequence space}. The space $(M_\psi(\mathbb N), \|\cdot\|_{M_\psi})$ is a fully symmetric sequence space; see, for example, \cite[Ch.\,II, \S\,5]{kps}, \cite[Part III, Ch.\,9, \S\,9.1]{rgmp}. In addition,
$\mathbf 1\notin M_\psi(\mathbb N)$ if and only if $\lim\limits_{t\to\ii}\frac{\psi(t)}{t}=0$ \cite[Ch.\,II, \S\,5]{kps}). Besides, $M_\psi(\mathbb N)= l^1$, as sets, if and only if $\psi(\ii)<\ii$.

If $\psi(+0)=0$, $\psi(\ii) =\ii$, and $\lim\limits_{t\to +0}\frac{\psi(t)}t=\ii$, then $M_\varphi$ is non-separable; see \cite{as}, \cite[Ch.\,II \S\,5, Lemma 5.4]{kps}.

If $\lim\limits_{t\to\ii}\frac{\psi(t)}{t}=0$, then $M_\psi(\mathbb N)\su c_0$, and we define
\[
\mc C_{M_\psi}=\mc C_{M_\psi(\mathbb N)},\ \ \| x\|_{\mc C_{M_\psi}}=\| x\|_{\mc C_{M_\psi(\mathbb N)}}, \ x \in \mc C_{M_\psi}.
\]
Finally, Theorems \ref{t12} and \ref{t14} imply the  following.
\begin{teo}\label{t53} Let $\psi$ be a concave function on $[0, \ii)$ with $\psi(0) = 0$, $\psi(t)>0$ for all $t > 0$,
and let $\lim\limits_{t\to\ii}\frac{\psi(t)}{t}=0$. Then
\begin{enumerate}[(i)]
\item
$(C_{M_\psi}\in (IET)$;
\item
if either $\psi(\ii)=\ii$, $\psi(+0) = 0$, and $\lim\limits_{t \to 0^+}\frac{\psi(t)}t=\ii$ or $\psi(\ii)<\ii$, then\\
$(\mc C_{M_\psi}, \|\cdot\|_{\mc C_{M_\psi}})\notin (MET)$.
\end{enumerate}
\end{teo}

\end{document}